\newcommand{\bbC}{\mathbb{C}}

\newcommand{\bbQ}{\mathbb{Q}}

\newcommand{\bbZ}{\mathbb{Z}}

\newcommand{\Tr}{\textup{Tr}}

\documentclass[11pt]{amsart}
\usepackage{amscd, amsmath, amsthm, amssymb}
\newtheorem{theorem}{Theorem}[section]
\newtheorem{lemma}[theorem]{Lemma}
\newtheorem{proposition}[theorem]{Proposition}

\theoremstyle{definition}     

\newtheorem{remark}[theorem]{Remark}
\numberwithin{equation}{section}

\begin{document}

\title[ K3 surfaces with an order 50 automorphism]{ K3 surfaces with an order 50 automorphism}

\author[J. Keum]{JongHae Keum}
\address{School of Mathematics, Korea Institute for Advanced Study, Seoul 130-722, Korea } \email{jhkeum@kias.re.kr}
\thanks{Research supported by National Research Foundation of Korea (NRF grant).}


\date{Dec 2013}
\begin{abstract} In any characteristic $p$ different from 2 and 5, Kond\=o gave an example of a K3 surface with a 
purely non-symplectic automorphism of order 50. The surface was explicitly given as a double plane branched along a smooth sextic curve. In this note we show that, in any characteristic $p\neq 2, 5$, a K3 surface with a cyclic action of order 50 is isomorphic to the example of  Kond\=o. 
\end{abstract}

\maketitle

Let $X$ be a K3 surface over an algebraically closed field $k$ of
characteristic $p\ge 0$. An automorphism $g$ of $X$ is called
\emph{symplectic} if it preserves any regular 2-form on $X$, and
 \emph{purely non-symplectic} if no power of $g$ is symplectic except the identity.

In any characteristic $p\ge 0$, $p\neq 2, 5$, Kond\=o \cite{Ko} gave an example of a K3 surface $X_{50}$ with a 
purely non-symplectic automorphism $g_{50}$ of order 50:
\begin{equation}\label{formula}
X_{50}= (w^2=x^6+xy^5+yz^5)\subset {\bf P}(1,1,1,3),
 \end{equation}
\begin{equation}\label{form}
 g_{50}(x,\,y,\, z,\, w)=(x,\, \zeta_{50}^{40}y,\, \zeta_{50}^{2}z,\,
\zeta_{50}^{25}w)
 \end{equation}
 where $\zeta_{50}\in k$ is a primitive 50th  root of unity. 
In characteristic $p=2$ and 5 the automorphism degenerates and the equation does not even define a K3 surface.

The result of this short note is the following.

\begin{theorem} \label{main} Let $k$ be an algebraically closed
field of characteristic $p\neq 2$, $5$. Let $X$ be a K3 surface
defined over $k$ with an automorphism $g$ of order $50$. Then
\begin{enumerate}
\item $g$ is purely non-symplectic;
\item the pair $(X, \langle g\rangle)$ is isomorphic to the pair $(X_{50},  \langle g_{50}  \rangle)$.
\end{enumerate}
\end{theorem}

The first statement of  Theorem \ref{main} was proved in a previous paper \cite[Lemmas 4.2 and 4.7]{K}.

Over $k=\bbC$ the second statement of Theorem \ref{main} was proved by Machida and Oguiso \cite{MO}, under the assumption that $g$ is purely non-symplectic. Our proof is characteristic free, does not use lattice theory and the holomorphic Lefschetz formula.

A similar characterization of K3 surfaces with a tame cyclic action of order 60 (resp. 66) was given in \cite{K60} (resp. \cite{K66}), where it was proven that for such a pair $(X, \langle g\rangle)$ the K3 surface $X$ admits a $g$-invariant elliptic fibration, thus can be given by a $g$-invariant Weierstrass equation. The case of order 50 is similar to the case or order 40 in \cite{K40}, and the K3 surface admits a $g$-invariant double plane presentation. 

\begin{remark} (1) In characteristic 5 it was shown in  the previous paper \cite[Main Theorem and Lemma 9.6]{K} that no K3 surface admits a cyclic action of order 25.

(2) In characteristic 2  there is a K3 surface with a cyclic action of order 50: 
\begin{equation}\label{formula2}
Y=(w^2+x^3w=x^6+xy^5+yz^5)\subset {\bf P}(1,1,1,3),
 \end{equation}
\begin{equation}\label{form2}
 f_{50}(x,\,y,\, z,\, w)=(x,\, \zeta_{25}^{20}y,\, \zeta_{25}^{}z,\,w+x^3)
 \end{equation}
  where $\zeta_{25}\in k$ is a primitive 25th  root of unity. Is this the unique pair up to isomorphism in characteristic 2? 
\end{remark}

\bigskip

{\bf Notation}

\bigskip
For a variety $X$ with an automorphism $g$, we use the following notation:
\medskip
\begin{itemize} 
\item ${\rm NS}(X)$ : the N\'eron-Severi group of $X$;
\medskip
\item $X^g={\rm Fix}(g)$ : the fixed locus of $g$ in $X$;
\medskip
\item $e(g):=e({\rm Fix}(g))$, the Euler characteristic of ${\rm Fix}(g)$ for $g$ tame;
\medskip
\item $\Tr(g^*|H^*(X)):=\sum_{j=0}^{2\dim X} (-1)^j\Tr (g^*|H^j_{\rm et}(X,{\bbQ}_l))$;
\medskip
\item $[g^*]=[\lambda_1, \ldots, \lambda_{b_2}]$ : the list of eigenvalues of $g^*|H^2_{\rm et}(X,{\bbQ}_l)$ where $b_2$ is the second Betti number of $X$;
\medskip
\item $\zeta_a$ : a primitive $a$-th  root of unity in $\overline{\bbQ_l}$;
\medskip
\item $[\zeta_a:\phi(a)]\subset [g^*]$ : all primitive $a$-th  roots of unity appear in $[g^*]$ where $\phi(a)$ indicates the number of
them;
\medskip
\item $[\lambda.r]\subset [g^*]$ : the eigenvalue $\lambda$ repeats $r$ times in
$[g^*]$;
\medskip
\item $[(\zeta_a:\phi(a)).r]\subset [g^*]$ : the list $\zeta_a:\phi(a)$ repeats $r$ times in
$[g^*]$.
\end{itemize}

\section{Preliminaries}

 We first recall the following basic
result used in the paper \cite{K}.

\begin{proposition}\label{integral}$($See  \cite[Proposition 2.1]{K}.$)$ Let $g$ be an automorphism of a projective variety $X$ over an algebraically closed field $k$ of
characteristic $p> 0$. Let $l$ be a prime $\neq p$. Then the following hold true.
\begin{enumerate}
\item $($See \cite [3.7.3]{Illusie}.$)$ The characteristic polynomial of
$g^*|H_{\rm et}^j(X,\bbQ_l)$ has integer coefficients for each
$j$. The characteristic polynomial does not depend on the choice of  cohomology, $l$-adic or crystalline. In particular, if a primitive
$m$-th root of unity appears with multiplicity $r$ as an
eigenvalue of $g^*|H_{\rm et}^j(X,\bbQ_l)$, then so does each of
its conjugates.
\item  If $g$ is of finite order, then $g$ has an
invariant ample divisor, and
$g^*|H_{\rm et}^2(X,\bbQ_l)$ has $1$ as an eigenvalue.
\item  If $X$ is a K3 surface, $g$ is tame and $g^*|H^0(X,\Omega_X^2)$ has
$\zeta_n\in k$ as an eigenvalue, then $g^*|H_{\rm et}^2(X,\bbQ_l)$
has $\zeta_n\in \overline{\bbQ_l}$ as an eigenvalue.
\end{enumerate}
\end{proposition}

The following is well known, see e.g., \cite[Theorem 3.2]{DL}.

\begin{proposition}\label{trace}$($Lefschetz fixed point formula$)$   Let $X$ be a smooth projective variety over  an algebraically closed field $k$ of
characteristic $p> 0$ and let $g$ be a tame automorphism of $X$. Then $X^g={\rm Fix}(g)$ is smooth and
 $$e(g):=e(X^g)=\Tr(g^*|
H^*(X)).$$
\end{proposition}

\begin{lemma}\label{nsym2}$($See \cite[Lemma 1.6]{K60}.$)$   
Let $X$ be a K3 surface in characteristic $p\neq 2$, admitting an automorphism $h$ of order $2$ with $\dim
H^2_{\rm et}(X,{\bbQ}_l)^h =2$. Then $h$ is non-symplectic and has
an $h$-invariant elliptic fibration $\psi:X\to {\bf P}^1$,  $$X/\langle h\rangle\cong {\bf F}_e$$ a rational ruled surface, and $X^h$ is either a curve of genus
$9$ which is a $4$-section of $\psi$ or the union of a section and a curve of genus $10$ which
is a $3$-section.
In the first case $e=0, 1$ or $2$, and in the second $e=4$. Each singular fibre of $\psi$
is of type $I_1$ $($nodal$)$, $I_2$, $II$ $($cuspidal$)$ or $III$, and is intersected by $X^h$ at the node and two smooth points if  of type $I_1$, at the two singular points if of type $I_2$, at the cusp with multiplicity $3$ and a smooth point if of type $II$, at the singular point tangentially to both components if of type $III$. If $X^h$ contains a section, then each singular fibre
is of type $I_1$ or $II$.
\end{lemma}

\begin{remark} If $e\neq 0$, the $h$-invariant elliptic fibration $\psi$ is the pull-back of the unique ruling of ${\bf F}_e$. If $e=0$, either ruling of ${\bf F}_0$ lifts to an $h$-invariant elliptic fibration.
\end{remark}

The following easy lemmas also will be used frequently.

\begin{lemma}\label{fix}$($See \cite[Lemma 2.10]{K}.$)$ Let $S$ be a set and ${\rm Aut}(S)$ be the group of bijections of $S$. For any $g\in {\rm Aut}(S)$ and positive integers $a$ and $b$,
\begin{enumerate}
\item ${\rm Fix}(g)\subset {\rm Fix}(g^a)$;
\item ${\rm Fix}(g^a)\cap {\rm Fix}(g^b)={\rm Fix}(g^d)$ where $d=\gcd (a, b)$;
\item ${\rm Fix}(g)= {\rm Fix}(g^a)$ if ${\rm ord}(g)$ is finite and  prime to $a$.
\end{enumerate}
\end{lemma}

\begin{lemma}\label{sum}$($See \cite[Lemma 2.11]{K}.$)$ Let $R(n)$ be the sum of all primitive $n$-th root of unity in $\overline{\bbQ}$ or in
$\overline{\bbQ_l}$, where $\gcd(l,n)=1$. Then
$$R(n)=\left\{\begin{array}{ccl} 0&{\rm if}& n\,{\rm has\,\, a\,\, square\,\, factor},\\
(-1)^t&{\rm if}& n\,{\rm is\,\, a\,\, product\,\, of}\,\,t\,\,{\rm distinct\,\, primes}.\\
\end{array} \right.$$
\end{lemma}

For an automorphism $g$ of finite order of a K3 surface $X$, tame or wild, we
write
$${\rm ord}(g)=m.n$$ if $g$ is of order $mn$ and the
natural homomorphism $\langle g\rangle\to {\rm GL}(H^0(X,
\Omega^2_X))$ has kernel of order $m$ and image of order $n$.

\section{Proof: the Tame Case}

Throughout this section, we assume that the characteristic $p>0$,
$p\neq 2$, 5. Let $g$ be an automorphism of order $50$ of a K3
surface. We first determine the list of eigenvalues of $g^*$ acting on the second cohomology of $X$.

\begin{lemma}\label{no1} $[g^{*}]\neq [1,\,1,\,
\zeta_{50}:20]$.
\end{lemma}

\begin{proof}
Suppose that $[g^{*}]= [1,\,1,\,
\zeta_{50}:20]$. Then $$[g^{25*}]= [1,\,\,1,\,
-1.20], \quad e(g^{25})=-16.$$ Thus ${\rm Fix}(g^{25})$ is either a curve $C_9$ of genus
$9$ or the union of a smooth rational curve and a curve $C_{10}$ of genus $10$. Using Lemma \ref{sum}, we compute $e(g)=4$ and
$$[g^{5*}]= [1,\,1,\,(\zeta_{10}:4).5], \quad e(g^{5})=9.$$ Note that $${\rm Fix}(g^{d})\subset {\rm Fix}(g^{25})$$ for any $d$ dividing 25. If  ${\rm Fix}(g^{25})$ is a curve $C_9$ of genus $9$, then
the action of $g$ (resp. $g^5$) on $C_9$ has 4 (resp. 9) fixed points, hence the degree 25 cover $C_9\to C_9/\langle g^{} \rangle$ has 4 points of ramification index 25 and 5 points of ramification index 5, which contradicts the Hurwitz formula.
By Lemma \ref{nsym2}, the quotient surface $X/\langle g^{25} \rangle$ is isomorphic to the rational ruled surface ${\bf F}_4$,  $X$ has a $g^{25}$-invariant elliptic
fibration $$\psi:X\to {\bf P}^1$$ and
${\rm Fix}(g^{25})$ is the union of a section $R$ and a curve $C_{10}$ of genus $10$ which
is a $3$-section of $\psi$. We also know that a fibre of $\psi$ is of type $I_0$ (smooth), $I_1$ or $II$.
The automorphism $\bar{g}$ of ${\bf F}_4$ induced by $g$ preserves the unique ruling, so $g$ preserves the elliptic fibration. 
Since $\bar{g}^{25}$ acts trivially on ${\bf F}_4$,
$g^{25}$ acts trivially on the base ${\bf P}^1$ and the orbit of a singular fibre under the action of  $g|{\bf P}^1$ has length 1 or 5. Thus $g^5|{\bf P}^1$ fixes all singular fibres and $g^5$ fixes the singular points of all singular fibres. Since $\psi$ has at least 12 singular fibres, $g^5$ fixes at least 12 points, contradicting $e(g^5)=9$.
\end{proof}

\begin{lemma}\label{50}
\begin{enumerate}
\item $[g^{*}]=[1,\,-1,\,\zeta_{50}:20]$\\ where the first eigenvalue corresponds to a $g^*$-invariant ample
class;
\item ${\rm Fix}(g^{25})=C_{10}$, a curve of genus $10$;
\item ${\rm Fix}(g^{10})=D_2\cup \{\,q\}$ where $D_{2}$ is a curve of genus $2$ intersecting $C_{10}$ at $6$ points, say $p_1,\ldots, p_6$, and $q$ is a point on $C_{10}$, but not on $D_2$;
\item ${\rm Fix}(g^{5})= \{\,p_1,\ldots, p_6,\, q \}$;
\item $g$ fixes one of the six points $p_i$, say $p_6$,  rotates the remaining 5, and ${\rm Fix}(g^{})= \{\, p_6,\, q \}$.
\end{enumerate}
\end{lemma}

\begin{proof} (1) By \cite[Lemmas 4.2 and 4.7]{K}, $g$ cannot be of order $2.25$ or $5.10$, hence is purely non-symplectic. By Proposition
\ref{integral} the action of $g^*$ on $H_{\rm et}^2(X,\bbQ_l)$ has
$\zeta_{50}\in \overline{\bbQ_l}$ as an eigenvalue. Thus
$[g^{*}]=[1,\,\pm 1,\, \zeta_{50}:20]$ and the result follows from Lemma \ref{no1}. 

(2) follows from (1), since $e(g^{25})=-18$ and the  invariant subspace of $g^*|H_{\rm et}^2(X,\bbQ_l)$ has dimension 1.

(3)-(5) We compute 
$$[g^{5*}]= [1,\,-1,\,(\zeta_{10}:4).5], \quad e(g^{5})=7,$$
$$[g^{10*}]= [1,\,1,\,(\zeta_{5}:4).5], \quad e(g^{10})=-1.$$ Since $e(g^{10})<0$,  
${\rm Fix}(g^{10})$ contains a curve of genus $>1$. Since the  invariant subspace of $g^{10*}|H_{\rm et}^2(X,\bbQ_l)$ has dimension 2, ${\rm Fix}(g^{10})$ contains at most one smooth rational curve. Since $e(g^5)=7$, ${\rm Fix}(g^{5})$ consists of 7 points of $C_{10}={\rm Fix}(g^{25})$. Suppose  ${\rm Fix}(g^{10})$ contains a rational curve $R$. Then 
$${\rm Fix}(g^{10})=R\cup D_{d+3}\cup\{ 2d+1\,\,{\rm points}\},\,\,d\ge 0.$$  Since
$C_{10}\cap D_{d+3}\subset {\rm Fix}(g^{25})\cap{\rm Fix}(g^{10})={\rm Fix}(g^{5}),$ we have
$$C_{10}D_{d+3}\le 7,$$ then by Hodge index theorem
$$18(2d+4)=C_{10}^2D_{d+3}^2\le(C_{10}D_{d+3})^2\le 7^2,$$ hence $d<0$, absurd. Thus ${\rm Fix}(g^{10})$ cannot contain a rational curve and
$${\rm Fix}(g^{10})=D_{d+2}\cup\{ 2d+1\,\,{\rm points}\},\,\,d\ge 0.$$ In the same way as above, we see that
$C_{10}D_{d+2}\le 7$. Then by the Hodge index theorem 
$$18(2d+2)=C_{10}^2D_{d+2}^2\le(C_{10}D_{d+2})^2\le 7^2,$$ hence $d=0$ and $6\le C_{10}D_{2}$.  Let $$q\in {\rm Fix}(g^{10})$$  be the isolated point. Then $g(q)=q$  and $g^5$ fixes 6 points on $D_2$. Then $C_{10}D_{2}\le 6$ as
$C_{10}\cap D_{2}\subset {\rm Fix}(g^{25})\cap{\rm Fix}(g^{10})={\rm Fix}(g^{5})$. Thus $C_{10}D_{2}=6$. Let $p_1,\ldots,p_6$ be the 6 intersection points of $C_{10}$ and $D_{2}$. Then $g^5$ fixes the 7 points, $p_1,\ldots,p_6$ and $q$. This proves (3) and (4). Since $e(g)=2$, the action of $g$ on $\{\,p_1,\ldots,p_6\,\}$
fixes one and rotates five, proving (5)
\end{proof}

\bigskip

\noindent {\bf Proof of the second statement of Theorem \ref{main}.} 

\bigskip
Lemma \ref{50} plays a key role in the proof. We modify the proof of \cite[Section 4]{MO}.
The quotient
$$X/\langle g^{25}\rangle$$ is a smooth rational surface with Picard number 1, hence is isomorphic to ${\bf P}^2$. The branch curve 
 $$B:=\bar{C}_{10}\subset {\bf P}^2$$ is a smooth sextic and the image
$$L:=\bar{D}_{2}\subset {\bf P}^2$$ of $D_2$ is a line.
Let $$\bar{p}_i, \,\,\bar{q}\in {\bf P}^2$$ be the images of $p_i$ and $q$.
Choose coordinates $x, y, z$ of ${\bf P}^2$ such that
$$L= (z=0), \quad \bar{q}=(0,0,1).$$
Our automorphism $g$ induces an automorphism $\bar{g}$ of ${\bf P}^2$. The fixed locus of $\bar{g}^{5}$
is the image of 
$${\rm Fix}(g^{5})\cup {\rm Fix}(g^{20})={\rm Fix}(g^{20})={\rm Fix}(g^{10}),$$ thus $${\rm Fix}(\bar{g}^{5})=L\cup \{\bar{q}\}.$$ The fixed locus of $\bar{g}^{}$
is the image of 
$${\rm Fix}(g^{})\cup {\rm Fix}(g^{24})={\rm Fix}(g^{24})={\rm Fix}(g^{2})=\{\, p_6,\, q, \,q_1,\,q_2 \} $$ where $q_1$ and $q_2$ are two points of $D_2$ which are interchanged by $g$, thus $${\rm Fix}(\bar{g}^{})=\{\bar{p}_6,\,\bar{q},\,\bar{q}_1 \}.$$ We further may assume that 
$$\bar{p}_6=(0,1,0),\quad\bar{q}_1=\bar{q}_2=(1,0,0).$$
From these, we infer that
$$\bar{g}(x,\,y,\,z)=(x,\, \zeta_{25}^{20}y,\, \zeta_{25}^{j}z)$$ for some $j$. Since $\bar{g}$ has order 25, $5\nmid j$.  The monomials $x^6$ and $xy^5$ are $\bar{g}$-invariant. We know that the branch $B=\bar{C}_{10}$ is a smooth sextic. Thus there must exist a $\bar{g}$-invariant monomial of the form $y^az^{6-a}$. Then $20a+j(6-a)\equiv 0$ mod 25. Since $5\nmid j$, $a=1$ and $j=5i+1$ for some $i$. Then 
$$(\zeta_{25}^{j})^{20}=(\zeta_{25}^{5i+1})^{20}=\zeta_{25}^{20},$$ we may assume that $j=1$. The branch $B$ is defined by $x^6+\alpha xy^5+\beta yz^5=0$ for some non-zero $\alpha$ and $\beta$. Replacing $y$ and $z$ by a scalar multiple, we may assume that  $\alpha=\beta=1$. Now the surface $X$ and the automorphism $g$ are defined by 
$$X: w^2=x^6+ xy^5+ yz^5$$
$$g(x,\,y,\, z,\, w)=(x,\, \zeta_{25}^{20}y,\, \zeta_{25}z,\,-w).$$

\section{Proof: the Complex Case}

We may assume that $X$ is projective, since a non-projective complex K3 surface cannot admit a non-symplectic
automorphism of finite order (\cite{Ueno}, \cite{Nik}) and
its automorphisms of finite order are symplectic, hence of order
$\le 8$.  Now the same  proof goes, once $H^2_{\rm
et}(X,\bbQ_l)$ is replaced by $H^2(X,\bbZ)$
and Proposition \ref{trace} by the usual topological Lefschetz
fixed point formula. 




\end{document}